\documentclass[oneside,A4,11pt]{amsart}
\usepackage{enumitem}
\usepackage{graphicx}
\usepackage[margin=1.4in]{geometry} 

\usepackage{amssymb}
\usepackage{amsthm}

\usepackage{amsmath}
\usepackage{color}
\usepackage[utf8]{inputenc}

\usepackage{tikz}
\usepackage[all]{xy}
\usepackage[bookmarksnumbered,colorlinks]{hyperref}
\usepackage{dsfont}
\usepackage[normalem]{ulem}
\usepackage{float} 
\def\br#1\er{\textcolor{red}{#1}} 
\usepackage{autonum}

\makeatletter
\@namedef{subjclassname@2020}{%
	\textup{2020} Mathematics Subject Classification}
\makeatother

\begin{document}
\title{Willmore surfaces and Hopf tori in homogeneous $3$-manifolds}
\author[A.L. Albujer]{Alma L. Albujer} \address{Departamento de
  Matemáticas, Edificio Albert Einstein\hfill\break\indent Universidad
  de Córdoba, Campus de Rabanales,\hfill\break\indent 14071 Córdoba,
  Spain}
 \email{aalbujer@uco.es}

\author[F.R. dos Santos]{F\'abio R. dos Santos} \address{Departamento de Matem\'atica\hfill\break\indent Universidade Federal de Pernambuco \hfill\break\indent 50.740-560, Recife, Pernambuco, Brazil}
 \email{fabio.reis@ufpe.br}

\newtheorem{thm}{Theorem}[section]
\newtheorem{theorem}{Theorem}[section]
\newtheorem{proposition}[thm]{Proposition} \newtheorem{lemma}[thm]{Lemma}
\newtheorem{corollary}[thm]{Corollary} \newtheorem{conv}[thm]{Convention}
\theoremstyle{definition} \newtheorem{definition}[thm]{Definition}
\newtheorem{notation}[thm]{Notation} \newtheorem{exe}[thm]{Example}
\newtheorem{conjecture}[thm]{Conjecture} \newtheorem{prob}[thm]{Problem}
\newtheorem{remark}[thm]{Remark}
\newtheorem{example}[thm]{Example}
\newcommand{\nablat}{\overline{\nabla}}
\renewcommand{\div}{\mathrm{div}}
\begin{abstract}
	 Some classification results for closed surfaces in Berger spheres are presented. On the one hand, a Willmore functional for isometrically immersed surfaces into an homogeneous space $\mathbb{E}^{3}(\kappa,\tau)$ with isometry group of dimension $4$ is defined and its first variational formula is computed. Then, we characterize Clifford and Hopf tori as the only Willmore surfaces satifying a sharp Simons-type integral inequality. On the other hand, we also obtain some integral inequalities for closed surfaces with constant extrinsic curvature in $\mathbb{E}^3(\kappa,\tau)$, becoming equalities if and only if the surface is a Hopf torus in a Berger sphere. 
\end{abstract}



\maketitle
\usetikzlibrary{matrix}

\section{Introduction}\label{sec:1}

A classical problem in the theory of isometric immersions is to classify immersed surfaces into a space form of constant sectional curvature having either constant mean curvature or constant Gaussian curvature. In this direction, we can highlight the rigidity theorems due to Alexandrov~\cite{Alexandrov:56}, Liebmann~\cite{Liebmann:1899} and Hilbert~\cite{Hilbert:01} on surfaces of constant curvature as the most celebrated results in the theory of surfaces in the Euclidean space $\mathbb{R}^{3}$. For generalizations of these results we quote~\cite{Aledo:05,Montiel:91}. Besides that, we cannot fail to highlight the classical Hopf's theorem~\cite{Hopf:83} which characterizes totally umbilical spheres as the unique topological spheres of constant mean curvature immersed into a $3$-dimensional space form of constant sectional curvature.

A natural generalization of space forms are the so-called {\em homogeneous} spaces. A Riemannian manifold is said to be homogeneous if for any two points $p$ and $q$, there exists an isometry that maps $p$ into $q$. Geometrically, an homogeneous manifold seems the same everywhere. As it is well-known, simply connected $3$-dimensional Riemannian homogeneous spaces are classified. Such manifolds have an isometry group of dimension $6$, $4$ or $3$. When the dimension is $6$, they correspond to space forms. When the dimension is $3$, the manifold has the geometry of the Lie group ${\rm Sol}_{3}$. In the case where the dimension of the isometry group is $4$, such manifold fibers over a two-dimensional space form of constant sectional curvature $\kappa$, $\mathbb{M}^2(\kappa)$, and its fibers are the trajectories of a unit Killing vector field. These last manifolds are usually denoted by $\mathbb{E}^{3}(\kappa,\tau)$, where $\tau$ is the constant bundle curvature of the natural projection $\pi:\mathbb{E}^3(\kappa,\tau)\rightarrow \mathbb{M}^2(\kappa)$ and $\kappa\neq4\tau^{2}$. According to the constants $\kappa$ and $\tau$ we can classify such spaces. When $\tau=0$, $\mathbb{E}^{3}(\kappa,0)=\mathbb{M}^{2}(\kappa)\times\mathbb{R}$ where $\mathbb{M}^{2}(\kappa)$ is the sphere $\mathbb{S}^{2}(\kappa)$ of curvature $\kappa>0$ or the hyperbolic plane $\mathbb{H}^{2}(\kappa)$ of curvature $\kappa<0$. When $\tau\neq0$, $\mathbb{E}^{3}(\kappa,\tau)$ is a Berger sphere $\mathbb{S}^{3}_{b}(\kappa,\tau)$ if $\kappa>0$, a Heisenberg group ${\rm Nil}_{3}(\tau)$ if $\kappa=0$ or the universal cover of $PSL(2,\mathbb{R})$ if $\kappa<0$.

In the last years, the study of surfaces in homogeneous spaces with $4$-di\-men\-sio\-nal isometry group has attracted the attention of many geometers. We can say that this attention is due to the studies of Abresch, Rosenberg and Meeks which made possible great advances in the research in this area~\cite{Abresch:04,Meeks:04,Rosenberg:02}. Indeed, Abresch and Rosenberg~\cite{Abresch:04} discovered an holomorphic quadratic differential for surfaces with constant mean curvature in these spaces and solved the Hopf's theorem for them. Moreover, these spaces are also related to the eight geometries of Thurston~\cite{Thurston:97}. Furthermore, in~\cite{Galvez:08}, G\'alvez, Mart\'inez and Mira considered the study of the classical Bonnet problem for surfaces in the homogeneous $3$-manifolds $\mathbb{E}^3(\kappa,\tau)$. Later on, Rosenberg and Tribuzy showed in~\cite{Tribuzy:12} a rigidity result for a family of complete surfaces in an homogeneous space having the same positive extrinsic curvature and satisfying a certain condition. 

Some years ago, Hu, Lyu and Wang developed in~\cite{Hu:15} a Simons-type integral inequality for immersed minimal closed surfaces into the homogeneous space $\mathbb{E}^{3}(\kappa,\tau)$, the equality being satisfied if and only if the surface has parallel second fundamental form. When the homogeneous space $\mathbb{E}^{3}(\kappa,\tau)$ is the Berger sphere $\mathbb{S}^{3}_{b}(\kappa,\tau)$ $(\kappa\neq4\tau^{2})$, they showed that the equality holds if and only if the surface is the Clifford torus. We recall that the Clifford torus is the only minimal Hopf torus in the Berger sphere. Recently, P\'ampano has considered in~\cite{Pampano:20} a more general setting, where the ambient space is the total space of a Killing submersion. Specifically, he studies surface energies depending on the mean curvature, which extend the classical notion of Willmore energy. Furthermore, the author constructs critical tori for these energy functionals.

Concerning product spaces, even more recently the second author has studied in~\cite{dos Santos:18} immersed complete surfaces into a product space $\mathbb{M}^{2}(\kappa)\times\mathbb{R}$ with nonnegative constant extrinsic curvature. In this setting, he has shown that these surfaces must be either cylinders when $\kappa=-1$, or slices when $\kappa=1$. Our goal is, on the one hand, to present a Willmore functional for immersed closed surfaces into $\mathbb{E}^3(\kappa,\tau)$, to obtain its Euler-Lagrange equation, and as a consequence to present a characterization result for closed Willmore surfaces in $\mathbb{S}^3_b(\kappa,\tau)$ in terms of an integral inequality. On the other hand, we extend the technics developed in~\cite{dos Santos:18} to the study of immersed closed surfaces with constant extrinsic curvature into $\mathbb{E}^{3}(\kappa,\tau)$ $(\tau\neq0)$.

The outline of the paper goes as follows. In Section~\ref{sec:2} we describe some basic facts about surfaces in the homogeneous space $\mathbb{E}^{3}(\kappa,\tau)$ $(\tau\neq0)$ with isometry group of dimension~$4$, introducing some relevant families of surfaces in such homogeneous spaces. Later on, working with the Cheng-Yau's operator, we develop in Section~\ref{sec:3} a Simons-type formula for these surfaces (cf. Proposition~\ref{prop:2.1}), as well as a divergence type formula involving the Cheng-Yau's operator (cf. Lemma~\ref{divergencias:CY}). In Section~\ref{sec:4} we compute the Euler-Lagrange equation for the Willmore functional of an immersed closed surface into an homogeneous space $\mathbb{E}^{3}(\kappa,\tau)$ (cf. Proposition~\ref{lem:1}). As an application, we characterize Clifford and Hopf tori as the only Willmore surfaces satisfying a sharp Simons-type integral inequality (cf. Theorem~\ref{teo:Willmore surfaces}). In the last section, we consider closed surfaces with constant extrinsic curvature and we also obtain integral inequalities, becoming equalities if and only if the surface is a Hopf torus in a Berger sphere $\mathbb{S}_{b}^{3}(\kappa,\tau)$ (cf. Theorems~\ref{teo:1.1} and~\ref{teo:1.2}).

\section{Preliminaries}\label{sec:2}

In this section, we will introduce some basic facts and notations that will appear along the paper.

Let $\kappa$ and $\tau$ be real numbers. The region $\mathcal{D}$ of the Euclidean space $\mathbb{R}^{3}$ given by
\begin{equation}\label{eq:D}
	\mathcal{D}=\left\{
	\begin{array}{ccc}
		\mathbb{R}^{3}, & \mbox{if} & \kappa\geq0 \\
		\mathbb{D}(2/\sqrt{-\kappa})\times\mathbb{R}, & \mbox{if} & \kappa<0
	\end{array}
	\right. 
\end{equation}
and endowed with the homogeneous Riemannian metric
\begin{equation}\label{eq:metric_E}
	\langle\,,\rangle_{R}=\lambda^{2}(dx^{2}+dy^{2})+\left(dz+\tau\lambda(ydx-xdy)\right)^{2},\quad\lambda=\dfrac{1}{1+\frac{\kappa}{4}(x^{2}+y^{2})},
\end{equation}
is the so-called {\it Bianchi-Cartan-Vranceanu space} ($BCV$-space) which is usually denoted by $\mathbb{E}^{3}(\kappa,\tau):=\left(\mathcal{D},\langle\,,\rangle_{R}\right)$.

As it is well-known, there exists a Riemannian submersion $\pi:\mathbb{E}^{3}(\kappa,\tau)\rightarrow \mathbb{M}^{2}(\kappa)$, where $\mathbb{M}^{2}(\kappa)$ is the $2$-dimensional simply connected space form of constant curvature $\kappa$, such that $\pi$ has constant bundle curvature $\tau$ and totally geodesic fibers. Furthermore, $\xi=E_3$ is a unit Killing field on $\mathfrak{X}(\mathbb{E}^{3}(\kappa,\tau))$ which is vertical with respect to $\pi$.

The $BCV$-spaces $\mathbb{E}^{3}(\kappa,\tau)$ are oriented, and then we can define a vectorial product $\wedge$, such that if $\{e_1,e_2\}$ are linearly independent vectors at a point $p$, then $\{e_1,e_2, e_1\wedge e_2\}$ determines an orientation at $p$. Then the properties of $\xi$ imply (see~\cite{Daniel:07}) that for any vector field $X$ on $\mathfrak{X}(\mathbb{E}^{3}(\kappa,\tau))$ the following relation holds
\begin{equation}\label{eq:xi}
	\overline{\nabla}_{X}\xi=\tau(X\wedge\xi),
\end{equation}
$\overline{\nabla}$ being the Levi-Civita connection of $\mathbb{E}^3(\kappa,\tau)$.
Moreover, let us recall that the curvature tensor of $\mathbb{E}^3(\kappa,\tau)$\footnote{We adopt for the $(1,3)$-curvature tensor of the spacetime the following definition (\cite[Chapter 3]{O'Neill:83}), $\overline{R}(X,Y)Z=\overline{\nabla}_{[X,Y]}Z-[\overline\nabla_X,\overline\nabla_Y]Z$.} satisfies, (see~\cite{Daniel:07}),
\begin{equation}\label{eq:curv_E}
	\begin{split}
		\overline{R}(X,Y)Z=&(\kappa-3\tau^{2})(\langle X,Z\rangle Y-\langle Y,Z\rangle X)\\
		&+(\kappa-4\tau^{2})\langle Z,\xi\rangle(\langle Y,\xi\rangle X-\langle X,\xi\rangle Y)\\
		&+(\kappa-4\tau^{2})(\langle Y,Z\rangle\langle X,\xi\rangle-\langle X,Z\rangle\langle Y,\xi\rangle)\xi,
	\end{split}
\end{equation}
where $X,Y,Z\in\mathfrak{X}(\mathbb{E}^3(\kappa,\tau))$.

In what follows, let $\Sigma^{2}$ be an isometrically immersed connected surface which we assume to be orientable and oriented by a globally defined unit normal vector field $N$. Let us denote by $A$ the second fundamental form of the immersion with respect to $N$ and by $\nabla$ the Levi-Civita connection of $\Sigma^{2}$. Then, the Gauss and Weingarten formulae are given by
\begin{equation}\label{eq:Gauss formula}
	\overline{\nabla}_XY=\nabla_XY+\langle A(X),Y\rangle N
\end{equation}
and
\begin{equation}\label{eq:Weingarten}
	A(X)=-\overline{\nabla}_XN,
\end{equation}
for every tangent vector fields $X,Y\in\mathfrak X(\Sigma)$.

Furthermore, we can consider a particular function naturally attached to such a surface $\Sigma^{2}$, namely, $C=\langle N,\xi\rangle$. Let us observe that $C$ measures the cosinus of the angle determined by the vector fields $N$ and $\xi$. A direct computation shows that the projection of the vector field $\xi$ on $\mathfrak{X}(\Sigma)$ is given by
\begin{equation}\label{eq:2.3}
	T=\xi^{\top}\!\!=\xi-CN,
\end{equation}
where $(\,\cdot)^{\top}$ denotes the tangential component of a vector field in $\mathfrak X(\mathbb{E}^3(\kappa,\tau))$ along $\Sigma^{2}$. Thus, we get
\begin{equation}\label{eq:2.4}
	|T|^2=1-C^2.
\end{equation}
Besides, from~\eqref{eq:xi},~\eqref{eq:2.3} and the Gauss and Weingarten formulae we easily obtain the integrability equations, 
\begin{equation}\label{eq:2.5}
	\nabla_{X}T=C (A-\tau J)(X)\quad\mbox{and}\quad\nabla C =-(A+\tau J) (T),
\end{equation}
where $J$ denotes the (oriented) rotation of angle $\pi/2$ on $T\Sigma$ given by $J(X)=N\wedge X$. In particular, \begin{equation}\label{eq:2.6}
	\langle J(X),J(Y)\rangle=\langle X,Y\rangle\quad\mbox{and}\quad J^{2}(X)=-X,
\end{equation}
for every $X,Y\in\mathfrak{X}(\Sigma)$. Therefore, from the first equation in~\eqref{eq:2.5} it easily follows that
\begin{equation}\label{eq:divT}
	{\rm div}(T)=2CH,
\end{equation}
where ${\rm div}$ denotes the divergence operator on $\Sigma^2$ and $H$ stands for the mean curvature of $\Sigma^2$, defined by $H=\frac{1}{2}{\rm tr}(A)$. Furthermore, it is immediate to check that 
\begin{equation}\label{eq:2.9}
	4H^{2}=|A|^{2}+2K_{e},
\end{equation}
where $|A|^2={\rm tr}(A^2)$ and $K_{e}={\rm det}(A)$ denotes the {\em extrinsic curvature} of $\Sigma^{2}$.

As it is well-known, the fundamental equations of $\Sigma^2$ are the {\em Gauss equation}
\begin{equation}\label{eq:2.10}
	\begin{split}
		R(X,Y)Z=&(\kappa-3\tau^{2})(\langle X,Z\rangle Y-\langle Y,Z\rangle X)\\
		&+(\kappa-4\tau^{2})\langle Z,T\rangle(\langle Y, T\rangle X-\langle X, T\rangle Y)\\
		&+(\kappa-4\tau^{2})(\langle Y,Z\rangle\langle X, T\rangle-\langle X,Z\rangle\langle Y, T\rangle) T\\
		&+\langle A(X),Z\rangle A(Y)-\,\langle A(Y),Z\rangle A(X),
	\end{split}
\end{equation}
where $R$ denote the curvature tensor of $\Sigma^{2}$ and $X,Y,Z\in\mathfrak{X}(\Sigma)$, and the {\em Codazzi equation}
\begin{equation}\label{eq:2.11}
	\nabla A(X,Y)-\nabla A(Y,X)=(\kappa-4\tau^{2})C(\langle X, T\rangle Y-\langle Y, T\rangle X),
\end{equation}
where $\nabla A:\mathfrak{X}(\Sigma)\times\mathfrak{X}(\Sigma)\longrightarrow\mathfrak{X}(\Sigma)$ denotes the covariant differential of $A$,
\begin{equation}\label{eq:2.14}
	\nabla A(X,Y)=(\nabla_{Y}A)(X)=\nabla_{Y}A(X)-A(\nabla_{Y}X),\quad\mbox{for all}\quad X,Y\in\mathfrak{X}(\Sigma).
\end{equation}
From the Gauss equation~\eqref{eq:2.10}, jointly with~\eqref{eq:2.4} and~\eqref{eq:2.9} it holds
\begin{equation}\label{eq:2.12}
	2K=2\tau^{2}+2(\kappa-4\tau^{2})C ^{2}+4H^{2}-|A|^{2}=2\tau^{2}+2(\kappa-4\tau^{2})C ^{2}+2K_{e}.
\end{equation}

Let us recall now some classical surfaces in $\mathbb{E}^{3}(\kappa,\tau)$ which can be constructed in the following way. Given any regular curve $\alpha$ in $\mathbb{M}^{2}(\kappa)$, $\pi^{-1}(\alpha)$ is an isometrically immersed surface into $\mathbb{E}^3(\kappa,\tau)$ which is usually known as a {\em Hopf cylinder}. Hopf cylinders are flat surfaces, which have $\xi$ as a parallel tangent vector field and they are characterized by $C=0$. Furthermore, these cylinders satisfy 
\begin{equation}\label{eq:Clifford}
	H=k_{g}/2,\quad K=0,\quad K_{e}=-\tau^{2}\quad\mbox{and}\quad |\Phi|^2=2H^2+2\tau^2,
\end{equation}
where $k_g$ is the geodesic curvature of $\alpha$. 

Moreover, if $\alpha$ is a closed curve and the Riemannian submersion $\pi$ has circular fibers, which happens just in the case where $\mathbb{E}^3(\kappa,\tau)$ is a Berger sphere $\mathbb{S}^{3}_{b}(\kappa,\tau)$, then $\pi^{-1}(\alpha)$ is a flat torus which is also called a {\em Hopf torus}. 

Let us remember at this point that the Berger sphere $\mathbb{S}_{b}^{3}(\kappa,\tau)$ is isometric to the usual sphere $\mathbb{S}^{3}=\{(z,w)\in\mathbb{C}^{2}\,;\,|z|^{2}+|w|^{2}=1\}$ endowed with the metric
\begin{equation}\label{eq:metric_Sb}
	\langle X,Y\rangle=\dfrac{4}{\kappa}\left(\langle X,Y\rangle_{\mathbb{S}^{3}}+\dfrac{1}{\kappa}\left(4\tau^{2}-\kappa\right)\langle X,V\rangle_{\mathbb{S}^{3}}\langle Y,V\rangle_{\mathbb{S}^{3}}\right),
\end{equation}
where $\langle\,,\rangle_{\mathbb{S}^{3}}$ stands for the usual metric on the sphere, $V_{(z,w)}=J(z,w)=(iz,iw)$ for each $(z,w)\in\mathbb{S}^{3}$ and $\kappa,\tau$ are real numbers with $\kappa>0$ and $\tau\neq0$. We note that if $\kappa=4\tau^{2}$ then $\mathbb{S}_{b}^{3}(\kappa,\tau)$ is, up to homotheties, the round sphere. The Hopf fibration $\pi:\mathbb{S}_{b}^{3}(\kappa,\tau)\rightarrow\mathbb{S}^{2}(\kappa)$, defined by
\begin{equation}\label{eq:Hopf_fib}
	\pi(z,w)=\dfrac{1}{\sqrt{\kappa}}\left(z\overline{w},\dfrac{1}{2}\left(|z|^{2}-|w|^{2}\right)\right),
\end{equation}
is a Riemannian submersion  whose fibers are geodesics. The vertical unit Killing vector field is given by $\xi=\dfrac{\kappa}{4\tau}V$. A particular Hopf torus in $\mathbb{S}_b^3(\kappa,\tau)$ is the {\em Clifford torus} given by $$\{(z,w)\in\mathbb{S}^{3}_{b}(\kappa,\tau)\,;\,|z|^{2}=|w|^{2}=1/2\}.$$
It is well-known that the Clifford torus is the only minimal Hopf torus in any Berger sphere (see for instance~\cite{Torralbo:12.1}).

Let us finish this section by recalling a classification result for parallel surfaces in $\mathbb{E}^3(\kappa,\tau)$, proved by Belkhelfa, Dillen and Inoguchi in~\cite{Dillen:02}. From now on, we will understand by a parallel surface a surface with parallel second fundamental form.

\begin{lemma}\label{lem:4.1}
	\cite[Theorem 8.2]{Dillen:02} Let $\Sigma^{2}$ be an isometrically immersed parallel surface into the homogeneous space $\mathbb{E}^{3}(\kappa,\tau)$, $\kappa-4\tau^2\neq 0$. Then,
	\begin{enumerate}
		\item if $\tau\neq 0$ $\Sigma^{2}$ is a piece of a Hopf cylinder over a Riemannian circle in $\mathbb{M}^2(\kappa)$, that is, over a closed curve in $\mathbb{M}^2(\kappa)$ with constant geodesic curvature.
		\item if $\tau =0$ $\Sigma^2$ is either a piece of a slice in $\mathbb{M}^2(\kappa)\times\mathbb{R}$ or of a Hopf cylinder over a Riemannian circle in $\mathbb{M}^2(\kappa)$.
	\end{enumerate}
\end{lemma}

\section{A Simons-type formula for the Cheng-Yau operator in $\mathbb{E}^{3}(\kappa,\tau)$}\label{sec:3}

In consideration of the foregoing we are going to compute the Laplacian of $|A|^{2}$. First and foremost, we recall the following Weitzenb\"{o}ck formula (see for instance~\cite{Nomizu:69})
\begin{equation}\label{eq:2.13}
	\dfrac{1}{2}\Delta|A|^{2}=\dfrac{1}{2}\Delta\langle A,A\rangle=|\nabla A|^{2}+\langle\Delta\,A,A\rangle ,
\end{equation}
where $\Delta A:\mathfrak{X}(\Sigma)\longrightarrow\mathfrak{X}(\Sigma)$ is the rough Laplacian of the second fundamental form, that is,
\begin{equation}\label{eq:2.15}
	\Delta A(X)={\rm tr}\left(\nabla^{2}A(X,\cdot,\cdot)\right)=\sum_{i=1}^{2}\nabla^{2}A(X,e_{i},e_{i}),
\end{equation}
$\{e_{1},e_{2}\}$ being an orthonormal frame on $\mathfrak{X}(\Sigma)$ and $\nabla^{2}A(X,Y,Z)=\left(\nabla_{Z}\nabla A\right)(X,Y)$ for all $X,Y,Z\in\mathfrak{X}(\Sigma)$. In this setting, on the one hand we obtain from the Codazzi equation~\eqref{eq:2.11} and the integrability equations~\eqref{eq:2.5} the following symmetry in the two firsts variables of $\nabla^2A$,
\begin{equation}\label{eq:2.16}
	\begin{split}
		\nabla^{2}A(X,Y,Z)=&\nabla^{2}A(Y,X,Z)-(\kappa-4\tau^{2})\langle(A+\tau J)(T),Z\rangle\left(\langle X, T\rangle Y-\langle Y, T\rangle X\right)\\
		&+(\kappa-4\tau^{2})C^{2}\left(\langle X,(A-\tau J)(Z)\rangle Y-\langle Y,(A-\tau J)(Z)\rangle X\right).
	\end{split}
\end{equation}
On the other hand, it is not difficult to see that
\begin{equation}\label{eq:2.17}
	\nabla^{2}A(X,Y,Z)=\nabla^{2}A(X,Z,Y)+R(Y,Z)A(X)-A(R(Y,Z)X).
\end{equation}

Making $Y=Z=e_{i}$ in~\eqref{eq:2.16} and taking traces, we have
\begin{equation}\label{eq:2.18}
	\begin{split}
		\sum_{i=1}^2\nabla^{2}A(X,e_{i},e_{i})=&\sum_{i=1}^2 \nabla^{2}A(e_{i},X,e_{i})-(\kappa-4\tau^{2})C^{2}\left(2HX-(A+\tau J)(X)\right)\\&-(\kappa-4\tau^{2})\left(\langle X,T\rangle(A+\tau J)(T)-\langle A(T),T\rangle X\right).\\
	\end{split}
\end{equation}
Furthermore, from~\eqref{eq:2.17} it yields
\begin{equation}\label{eq:2.19}
	\nabla^{2}A(e_{i},X,e_{i})=\nabla^{2}A(e_{i},e_{i},X)+R(X,e_{i})A(e_{i})-A(R(X,e_{i})e_{i}).
\end{equation}
Observe now that, using the Gauss equation~\eqref{eq:2.10}, we get
\begin{equation}\label{eq:2.20}
	\begin{split}
		\sum_{i=1}^{2}R(X,e_{i})A(e_{i})=&(\kappa-3\tau^{2})(A(X)-2HX)-|A|^{2}A(X)+A^{3}(X)\\
		&+(\kappa-4\tau^{2})\left(\langle A( T), T\rangle X-\langle X, T\rangle A( T)\right)\\
		&+(\kappa-4\tau^{2})\left(2H\langle X, T\rangle-\langle A( T),X\rangle\right) T
	\end{split}
\end{equation}
and
\begin{equation}\label{eq:2.21}
	\sum_{i=1}^{2}A(R(X,e_{i})e_{i})=-(\kappa-3\tau^{2})A(X)+A^{3}(X)-2HA^{2}(X)+(\kappa-4\tau^{2})| T|^{2}A(X).
\end{equation}
Thus, inserting these two last equalities in~\eqref{eq:2.19},
\begin{equation}\label{eq:2.22}
	\begin{split}
		\sum_{i=1}^2\nabla^{2}A(e_{i},X,e_{i})&=\sum_{i=1}^2\nabla^{2}A(e_{i},e_{i},X)+2(\kappa-3\tau^{2})(A(X)-HX)+2HA^{2}(X)\\
		&+(\kappa-4\tau^{2})\left(\langle A( T), T\rangle X-\langle X, T\rangle A( T)-| T|^{2}A(X)\right)\\
		&+(\kappa-4\tau^{2})\left(2H\langle X, T\rangle-\langle A( T),X\rangle\right) T-|A|^{2}A(X).
	\end{split}
\end{equation}
Observe now that, since the trace commutes with the Levi-Civita connection,
\begin{equation}\label{eq:2.23}
	\sum_{i=1}^{2}\nabla^{2}A(e_{i},e_{i},X)={\rm tr}\left(\nabla_{X}\nabla A\right)=\nabla_{X}({\rm tr}(\nabla A)).
\end{equation}

We claim that 
\begin{equation}\label{eq:traza_nA}
	{\rm tr}(\nabla A)=2\nabla H+C(\kappa-4\tau^{2}) T.
\end{equation} 
Indeed, using the Codazzi equation~\eqref{eq:2.11},
\begin{equation}\label{eq:2.24}
	\begin{split}
		\langle\nabla A(e_{i},e_{i}),X\rangle&=\langle(\nabla_{e_{i}}A)(e_{i}),X\rangle=\langle e_{i},(\nabla_{e_{i}}A)(X)\rangle=\langle e_{i},\nabla A(X,e_{i})\rangle\\
		&=\langle e_{i},\nabla A(e_{i},X)\rangle+(\kappa-4\tau^{2})C (\langle X, T\rangle-\langle X, e_{i}\rangle\langle T,e_{i}\rangle),
	\end{split}
\end{equation}
which implies that
\begin{equation}\label{eq:tr_nablaA}
	\begin{split}
		\langle{\rm tr}(\nabla A),X\rangle
		&=\sum_{i=1}^{2}\left[\langle e_{i},(\nabla_{X}A)(e_{i})\rangle+C(\kappa-4\tau^{2})\left(\langle X, T\rangle-\langle X,e_{i}\rangle\langle T,e_{i}\rangle\right)\right]\\
		&=2\langle\nabla H,X\rangle+(\kappa-4\tau^{2})C\langle X, T\rangle,
	\end{split}
\end{equation}
for all $X\in\mathfrak{X}(\Sigma)$, so the claim is proved. Hence, from~\eqref{eq:traza_nA} and~\eqref{eq:2.5}, it holds
\begin{equation}\label{eq:2.26}
	\nabla_{X}({\rm tr}(\nabla A))=2\nabla_{X}\nabla H-(\kappa-4\tau^{2})\langle(A+\tau J)(T),X\rangle T+(\kappa-4\tau^{2})C^{2}(A-\tau J)(X).
\end{equation}

Therefore, putting~\eqref{eq:2.18}, \eqref{eq:2.22} and~\eqref{eq:2.26} in~\eqref{eq:2.15},  
\begin{equation}\label{eq:2.27}
	\begin{split}
		\Delta A(X)=&2\nabla_{X}\nabla H+2(\kappa-3\tau^{2})(A(X)-HX)-|A|^{2}A(X)+2HA^{2}(X)\\
		&+(\kappa-4\tau^{2})\left(2\langle A( T), T\rangle X-2\langle X, T\rangle A( T)-| T|^{2}A(X)\right)\\
		&+(\kappa-4\tau^{2})\left(2H\langle X, T\rangle-2\langle A( T),X\rangle\right) T+\,2(\kappa-4\tau^{2})C^{2}\left(A(X)-HX\right)\\
		&-\tau(\kappa-4\tau^2)\left(\langle J( T),X\rangle T+\langle X, T\rangle J( T)\right).
	\end{split}
\end{equation}
Consequently,
\begin{equation}\label{eq:2.28}
	\begin{split}
		\langle\Delta A,A\rangle=&2{\rm tr}(A\circ{\rm Hess}\,H)+2(\kappa-3\tau^{2})(|A|^{2}-2H^{2})+2(\kappa-4\tau^{2})C^{2}(|A|^{2}-2H^{2})\\
		&+2(\kappa-4\tau^{2})\left(3H\langle A( T), T\rangle-2\langle A^{2}( T), T\rangle-\tau\langle A( T),J( T)\rangle\right)\\
		&-(\kappa-4\tau^{2})| T|^{2}|A|^{2}-|A|^{4}+2H{\rm tr}(A^{3}).
	\end{split}
\end{equation}
Now, taking into account the characteristic polynomial of $A$, we observe that
\begin{equation}\label{eq:2.29}
	4H\langle A(T),T\rangle-2\langle A^{2}( T), T\rangle=2| T|^{2}K_{e}.
\end{equation}
Besides that, from~\eqref{eq:2.9} it holds
\begin{equation}\label{eq:2.30}
	\begin{split}
		2(\kappa-3\tau^{2})&(|A|^{2}-2H^{2})+(\kappa-4\tau^{2})(1-C^{2})(2K_{e}-|A|^{2})\\
		&=2(\kappa-3\tau^{2})(|A|^{2}-2H^{2})-2(\kappa-4\tau^{2})(1-C^{2})(|A|^{2}-2H^{2})\\
		&=2(|A|^{2}-2H^{2})(\tau^{2}+(\kappa-4\tau^{2})C^{2}).
	\end{split}
\end{equation}
Moreover, it is easy to check that ${\rm tr}(A^3)=3H|A|^2-4H^3$, so again from~\eqref{eq:2.9} we deduce that
\begin{equation}\label{eq:2.31}
	-|A|^{4}+2H{\rm tr}(A^{3})=-|A|^{4}+6H^{2}|A|^{2}-8H^{4}=2(|A|^{2}-2H^{2})K_{e}.
\end{equation}
Hence, taking into account~\eqref{eq:2.12},~\eqref{eq:2.29},~\eqref{eq:2.30} and~\eqref{eq:2.31},~\eqref{eq:2.28} reads
\begin{equation}\label{eq:2.31bis}
	\begin{split}
		\langle\Delta A,A\rangle&=2{\rm tr}(A\circ{\rm Hess}\,H)+2(|A|^{2}-2H^{2})K+2(\kappa-4\tau^{2})C^{2}(|A|^{2}-2H^{2})\\
		&\quad+2(\kappa-4\tau^{2})\left(H\langle A( T), T\rangle-\langle A^{2}( T), T\rangle-\tau\langle A( T),J( T)\rangle\right),
	\end{split}
\end{equation}
so~\eqref{eq:2.13} yields
\begin{equation}\label{eq:2.32}
	\begin{split}
		\dfrac{1}{2}\Delta|A|^{2}=&|\nabla A|^{2}+2{\rm tr}(A\circ{\rm Hess}\,H)\\
		&+2(|A|^{2}-2H^{2})K+2(\kappa-4\tau^{2})C^{2}(|A|^{2}-2H^{2})\\
		&+2(\kappa-4\tau^{2})\left(H\langle A( T), T\rangle-\langle A^{2}( T), T\rangle-\tau\langle A( T),J( T)\rangle\right).
	\end{split}
\end{equation}

\begin{remark}\label{rem:1.1}
	Let us observe that formula~\eqref{eq:2.32} was already obtained in~\cite{Hu:15}. In fact, let us consider a local orthonormal frame $\{e_{1},e_{2}\}$ such that $A(e_{1})=\lambda_{1}e_{1}$ and $A(e_{2})=\lambda_{2}e_{2}$. Moreover, by the definition of $J$, we must have $J(e_{1})=e_{2}$ and $J(e_{2})=-e_{1}$. Taking into account these two facts and writing $ T=\langle T\!\!,e_{1}\rangle e_{1}+\langle T\!\!,e_{2}\rangle e_{2}$, we have
	\begin{equation}\label{eq:remark}
		\langle A( T),J( T)\rangle=(\lambda_{2}-\lambda_{1})\langle T\!\!,e_{1}\rangle\langle T\!\!,e_{2}\rangle,
	\end{equation}
	so we recover~\cite[Lemma 3.1]{Hu:15}. However, we have included the proof for the sake of completeness, and because it represents an alternative reasoning based on tensorial analysis.
\end{remark}

Nevertheless, our aim in this section is to obtain a Simons-type formula for the Cheng-Yau's operator. To this respect, following~\cite{Cheng-Yau:77} we introduce the Cheng-Yau's operator $\square$ acting on any smooth function $u:\Sigma^{2}\rightarrow\mathbb{R}$ given by
\begin{equation}\label{eq:2.33}
	\square u={\rm tr}(P\circ{\rm Hess}\,u),
\end{equation}
where $P$ denote the first Newton transformation of $A$, that is, $P:\mathfrak{X}(\Sigma)\rightarrow \mathfrak{X}(\Sigma)$ is the operator given by
\begin{equation}\label{eq:2.34}
	P=2HI-A,
\end{equation}
which is also a self-adjoint linear operator which commutes with $A$ and satisfies ${\rm tr}(P)=2H$.

Taking $u=2H$, from equation~\eqref{eq:2.9} we obtain the following,
\begin{equation}\label{eq:2.35}
	\begin{split}
		\square(2H)&={\rm tr}(P\circ{\rm Hess}\,(2H))\\
		&=2H\Delta(2H)-2{\rm tr}(A\circ{\rm Hess}\,H)\\
		&=\dfrac{1}{2}\Delta(2H)^{2}-4|\nabla H|^{2}-2{\rm tr}(A\circ{\rm Hess}\,H)\\
		&=\dfrac{1}{2}\Delta|A|^{2}+\Delta K_{e}-4|\nabla H|^{2}-2{\rm tr}(A\circ{\rm Hess}\,H).
	\end{split}
\end{equation}
Inserting~\eqref{eq:2.32} in previous equality, we get
\begin{equation}\label{eq:2.36}
	\begin{split}
		\square(2H)=&\Delta K_{e}+|\nabla A|^{2}-4|\nabla H|^{2}+2(|A|^{2}-2H^2)K+2(\kappa-4\tau^{2})C^{2}\left(|A|^{2}-2H^{2}\right)\\
		&+2(\kappa-4\tau^{2})\left(H\langle A( T), T\rangle-\langle A^{2}( T), T\rangle-\tau\langle A( T),J( T)\rangle\right).
	\end{split}
\end{equation}

For our purpose, it will be more appropriate to deal with the traceless part of $A$, which is given by $\Phi=A-HI$, with $I$ the identity operator on $\mathfrak{X}(\Sigma)$. Then, ${\rm tr}(\Phi)=0$ and
\begin{equation}\label{eq:2.38}
	|\Phi|^{2}=|A|^{2}-2H^{2}\geq0,
\end{equation}
with equality at $p\in\Sigma^{2}$ if and only if $p$ is an umbilical point. In contrast to the case where the ambient is a Riemannian product, it was proved in~\cite{Toubiana:} that there does not exist any totally umbilical surface in $\mathbb{E}^{3}(\kappa,\tau)$ with $\tau\neq0$. 

Now, from the characteristic polynomial of $\Phi$ and identity~\eqref{eq:2.38}, the following equalities hold,
\begin{equation}\label{eq:2.39}
	-2\langle A^{2}(T),T\rangle+2H\langle A(T),T\rangle
	=-|\Phi|^{2}|T|^{2}-2H\langle\Phi(T),T\rangle
\end{equation}
and
\begin{equation}\label{eq:2.40}
	2C^{2}|A|^{2}-4H^{2}C^{2}=2C^{2}|\Phi|^{2}.
\end{equation}
Besides this, equations~\eqref{eq:2.4} and~\eqref{eq:2.12} give us
\begin{equation}\label{eq:2.41}
	2K+(\kappa-4\tau^{2})\left(2C^{2}-|T|^{2}\right)=2K_{e}+5(\kappa-4\tau^{2})C^{2}-\kappa+6\tau^{2}.
\end{equation}
Therefore, inserting these three last equations in~\eqref{eq:2.36}, we have finally shown the following Simons-type formula for the Cheng-Yau's operator.

\begin{proposition}\label{prop:2.1}
	Let $\Sigma^{2}$ be an isometrically immersed surface into an homogeneous space $\mathbb{E}^{3}(\kappa,\tau)$. Then,
	\begin{equation}
		\begin{split}
			\square(2H)&=\Delta K_{e}+|\nabla A|^{2}-4|\nabla H|^{2}+|\Phi|^{2}\left(2K_{e}+(\kappa-4\tau^{2})(5C^{2}-1)+2\tau^{2}\right)\nonumber\\
			&\quad-2(\kappa-4\tau^{2})\left(H\langle\Phi(T),T\rangle+\tau\langle\Phi(T),J(T)\rangle\right).
		\end{split}
	\end{equation}
\end{proposition}

\begin{remark}\label{rem:1.2}
	When $\tau=0$, as it was said in the Introduction, the homogeneous space $\mathbb{E}^{3}(\kappa,\tau)$ is exactly the product space $\mathbb{M}^{2}(\kappa)\times\mathbb{R}$, where $\mathbb{M}^{2}(\kappa)$ is a space form with constant sectional curvature $\kappa$. Thus, Proposition~\ref{prop:2.1} extends~\cite[Proposition 1.2]{dos Santos:18}.
\end{remark}

Let us finish this section by showing a nice divergence formula involving the Cheng-Yau's operator.

\begin{lemma}\label{divergencias:CY}
	Let $\Sigma^2$ be an isometrically immersed surface into an homogoneous space $\mathbb{E}^3(\kappa,\tau)$. Then,
	\begin{equation}\label{eq:div:CY}{\rm div}(P(2\nabla H))=\square(2H)-2C(\kappa-4\tau^{2})T(H).
	\end{equation}
\end{lemma}	
\begin{proof}
	Observe that by a standard tensor computation
	\begin{equation}\label{eq:3.6}
		{\rm div}(P(2\nabla H))=\square(2H)+2\langle{\rm div}\,P,\nabla H\rangle,
	\end{equation}
	where
	\begin{equation}\label{eq:3.7}
		{\rm div}\,(P)=\sum_{i=1}^{2}\nabla P(e_{i},e_{i})
	\end{equation}
	with
	\begin{equation}\label{eq:3.8}
		\nabla P(X,Y)=(\nabla_{Y}P)X=\nabla_{Y}(PX)-P(\nabla_{Y}X),
	\end{equation}
	for every $X,Y\in\mathfrak{X}(\Sigma)$. 
	
	It remains to compute the last term of equation~\eqref{eq:3.6}. Indeed, from~\eqref{eq:2.34},
	\begin{equation}\label{eq:3.9}
		\nabla P(X,Y)=2Y(H)X-\nabla A(X,Y),
	\end{equation}
	for every $X,Y\in\mathfrak{X}(\Sigma)$. Then,~\eqref{eq:traza_nA} implies that
	\begin{equation}\label{eq:3.10}
		{\rm div}\,(P)={\rm tr}(\nabla P)=2\nabla H-2\nabla H-(\kappa-4\tau^{2})C  T=-C(\kappa-4\tau^{2}) T,
	\end{equation}
	so finally~\eqref{eq:div:CY} follows from~\eqref{eq:3.6} and~\eqref{eq:3.10}.
\end{proof}

\section{Willmore surfaces in $\mathbb{S}_b^3(\kappa,\tau)$}\label{sec:4}

Let $x:\Sigma^{2}\rightarrow\mathbb{M}^3(\kappa)$ be an isometrically immersed orientable closed, i.e. compact without boundary, surface into the Riemannian space form $\mathbb{M}^{3}(\kappa)$ with constant sectional curvature $\kappa$. The Willmore functional is defined by
\begin{equation}\label{Willmore classic}
	\mathcal{W}(x)=\int_{\Sigma}(H^{2}+\kappa)dA,
\end{equation}
where $dA$ denotes the area element of the induced metric on $\Sigma^{2}$. Associated to this functional, there is the famous Willmore conjecture, solved in $2012$ by Marques and Neves~\cite{Coda:12}, which guarantees that this integral is at least $2\pi^{2}$ when $\Sigma^{2}$ is an immersed torus into $\mathbb{R}^{3}$. We say that $\Sigma^{2}$ is a {\em Willmore surface} if it is a stationary point for the functional $\mathcal{W}$. Moreover, it is well known that $\mathcal{W}$ is a conformal invariant and its Euler-Lagrange equation is given by (see~\cite{Bryant:84,Weiner:78})
\begin{equation}\label{eq:classic}
	\Delta H+|\Phi|^{2}H=0.
\end{equation}

For our interests, let $x:\Sigma^{2}\rightarrow\mathbb{E}^{3}(\kappa,\tau)$ be an isometrically immersed orientable closed surface into the homogeneous $3$-manifold $\mathbb{E}^{3}(\kappa,\tau)$. Following Weiner~\cite{Weiner:78}, we consider the following Willmore functional,
\begin{equation}\label{Willmore f}
	\mathcal{W}(x)=\int_{\Sigma}(H^{2}+\overline{K})dA,
\end{equation}
where at any $p\in\Sigma^2$, $\overline{K}$ denotes the sectional curvature of $T_p\Sigma$ in $\mathbb{E}^{3}(\kappa,\tau)$, which following~\eqref{eq:curv_E}-\eqref{eq:2.4} can be expressed as
\begin{equation}\label{eq:curvsect}
	\overline{K}=\tau^2+(\kappa-4\tau^2)C^2.
\end{equation}

In the following result we obtain the Euler-Lagrange equation of $\mathcal{W}$, extending the result of Weiner~\cite[Theorem 2.2]{Weiner:78} for immersed surfaces into the homogeneous space $\mathbb{E}^{3}(\kappa,\tau)$.

\begin{proposition}\label{lem:1}
	Let $x:\Sigma^{2}\rightarrow\mathbb{E}^{3}(\kappa,\tau)$ be an isometrically immersed orientable closed surface. Then $x$ is a stationary point of $\mathcal{W}$ if and only if
	\begin{equation}\label{eq:5.2}
		\Delta H+\left(|\Phi|^{2}+(\kappa-4\tau^{2})(1+C^{2})\right)\!H-2(\kappa-4\tau^{2})\langle A(T),T\rangle=0.
	\end{equation}
\end{proposition}

\begin{proof}
	Let us consider a variation of $x$, that is, a smooth map $X:(-\varepsilon,\varepsilon)\times\Sigma^2\rightarrow\mathbb{E}^{3}(\kappa,\tau)$ satisfying {that for each} $t\in(-\varepsilon,\varepsilon)$, the map $X_{t}:\Sigma^{2}\rightarrow\mathbb{E}^{3}(\kappa,\tau)$,
	given by $X_{t}(p)=X(t,p)$, is an immersion {and} $X_{0}=x$. Then, we can compute the first variation of $\mathcal{W}$ along $X$, that is,
	\begin{equation}\label{eq:5.3}
		\begin{split}
			\dfrac{d}{dt}\mathcal{W}(X_{t})\bigg|_{t=0}&=\dfrac{d}{d t}\int_{\Sigma}(H^{2}_{t}+\overline{K}_{t})dA_{t}\bigg|_{t=0}\\
			&=\int_{\Sigma}\left(\dfrac{d}{dt}(H^{2}_{t}+\overline{K}_{t})dA_{t}+(H^{2}_{t}+\overline{K}_{t})\dfrac{d}{d t}(dA_{t})\right)\bigg|_{t=0},
		\end{split}
	\end{equation}
	where, for each $t\in (-\varepsilon,\varepsilon)$, $H_t$ and $\overline{K}_t$ stand, respectively, for the mean curvature of $\Sigma^2$ and the sectional curvature of $T_p\Sigma$ in $\mathbb{E}^3(\kappa,\tau)$ with respect to the metric induced by $X_t$ and $dA_t$ denotes its volume element.
	
	Observe that, on the one hand, the following identity is well known (see for instance~\cite{do Carmo:88})
	\begin{equation}\label{eq:5.4}
		2\dfrac{dH_{t}}{dt}\bigg|_{t=0}=\Delta f+2\langle\nabla H,Y^{\top}\rangle+(\overline{\rm Ric}(N,N)+|A|^{2})\,f,
	\end{equation}
	{$\overline{\rm Ric}$ being the Ricci curvature tensor of $\mathbb{E}^3(\kappa,\tau)$ and $Y=\frac{\partial X}{\partial t}\big|_{t=0}$ the variational vector field related to the variation $X$, which can be decomposed as $Y=Y^\top+fN$ with $f=\langle Y,N\rangle$.}
	
	On the other hand, denoting by $N_t$ the unit normal vector field along $\Sigma^{2}$ with respect to the metric induced by $X_{t}$, since $N_0=N$ it holds
	\begin{equation}
		\dfrac{d\overline{K}_{t}}{d t}\bigg|_{t=0}=\dfrac{d}{dt}\left(\tau^{2}+(\kappa-4\tau^{2})\langle N_{t},\xi\rangle^{2}\right)\bigg|_{t=0}=2(\kappa-4\tau^{2})\langle N_{t},\xi\rangle\dfrac{d}{dt}\langle N_{t},\xi\rangle\bigg|_{t=0}.
	\end{equation}
	Since $Y=\frac{\partial X}{\partial t}\big|_{t=0}$ is a coordinate field, there exists an orthonormal frame $\{e_1,e_2\}$ in $\mathfrak{X}(\Sigma)$ such that $[Y,e_{k}]=0$, for any $k=1,2$. Thus, a direct computation gives us
	\begin{equation}
		\nabla f=-\overline{\nabla}_{Y}N-A(Y^{\top}).
	\end{equation}
	Then, from the integrability equations~\eqref{eq:2.5} we get
	\begin{equation}\label{eq:5.7}
		\dfrac{d\overline{K}_{t}}{dt}\bigg|_{t=0}=-2(\kappa-4\tau^{2})C\left(\langle\nabla f,T\rangle+\langle(A+\tau J)(T),Y^{\top}\rangle\right).
	\end{equation}
	Furthermore, by using Lemma $4.2$ of~\cite{Colares:97} (see also~\cite[Lemma 5.4]{Cao:07}), we have
	\begin{equation}
		\dfrac{d}{dt}(dA_{t})\bigg|_{t=0}=\left(-2Hf+{\rm div}(Y^{\top})\right)dA.
	\end{equation}
	Using the previous equalities we obtain
	\begin{equation}\label{eq:5.8}
		\begin{split}
			\dfrac{d}{dt}(H^{2}_{t}+\overline{K}_{t})\bigg|_{t=0}dA=&2H\dfrac{dH_{t}}{dt}\bigg|_{t=0}dA+\dfrac{d\overline{K}_{t}}{dt}\bigg|_{t=0}dA\\
			=&H\left(\Delta f+(\overline{\rm Ric}(N,N)+|A|^{2})f\right)dA+\langle\nabla H^{2},Y^{\top}\rangle dA\\
			&-2(\kappa-4\tau^{2})C\left(\langle\nabla f,T\rangle+\langle(A+\tau J)(T),Y^{\top}\rangle\right)dA
		\end{split}
	\end{equation}
	and
	\begin{equation}\label{eq:5.9}
		(H^{2}+\overline{K})\dfrac{d}{dt}(dA_{t})\bigg|_{t=0}=-2H(H^{2}+\overline{K})fdA+(H^{2}+\overline{K}){\rm div}(Y^{\top})dA.
	\end{equation}
	
	Let us also observe that
	\begin{equation}\label{eq:5.10}
		{\rm div}(H^{2}Y^{\top})=H^{2}{\rm div}(Y^{\top})+\langle\nabla H^{2},Y^{\top}\rangle.
	\end{equation}
	From~\eqref{eq:2.5} it also holds
	\begin{equation}\label{eq:5.11}
		{\rm div}(\overline{K}Y^{\top})=\overline{K}{\rm div}(Y^{\top})-2(\kappa-4\tau^{2})C\langle(A+\tau J)(T),Y^{\top}\rangle.
	\end{equation}
	Then, it follows from~\eqref{eq:5.9} that
	\begin{equation}\label{eq:5.12}
		\begin{split}
			(H^{2}+\overline{K})\dfrac{d}{dt}(dA_{t})\bigg|_{t=0}=&-2H(H^{2}+\overline{K})fdA+{\rm div}(H^{2}Y^{\top})dA-\langle\nabla H^{2},Y^{\top}\rangle dA\\
			&+{\rm div}(\overline{K}Y^{\top})dA+2(\kappa-4\tau^{2})C\langle(A+\tau J)(T),Y^{\top}\rangle dA.
		\end{split}
	\end{equation}
	Hence, replacing~\eqref{eq:5.8} and~\eqref{eq:5.12} in~\eqref{eq:5.3}, we get
	\begin{equation}\label{eq:5.13}
		\begin{split}
			\dfrac{d}{dt}\mathcal{W}(X_{t})\bigg|_{t=0}=&\int_{\Sigma}\left(H\Delta f+H(\overline{\rm Ric}(N,N)+|A|^{2})f-2H(H^{2}+\overline{K})f\right)dA\\
			&-2(\kappa-4\tau^{2})\int_{\Sigma}C\langle\nabla f,T\rangle dA\\
			=&\int_{\Sigma}\left(\Delta H+(\overline{\rm Ric}(N,N)+|A|^{2})H-2H(H^{2}+\overline{K})\right)fdA\\
			&-2(\kappa-4\tau^{2})\int_{\Sigma}C\langle\nabla f,T\rangle dA.
		\end{split}
	\end{equation}
	Besides this, from~\eqref{eq:2.5} and~\eqref{eq:divT},
	\begin{equation}\label{eq:5.11.1}
		\begin{split}
			{\rm div}(CfT)&=Cf{\rm div}(T)+C\langle\nabla f,T\rangle+f\langle\nabla C,T\rangle\\
			&=2HC^{2}f+C\langle\nabla f,T\rangle-f\langle A(T),T\rangle.
		\end{split}
	\end{equation}
	Therefore
	\begin{equation}\label{eq:5.13.1}
		\begin{split}
			\dfrac{d}{dt}\mathcal{W}(X_{t})\bigg|_{t=0}=&\int_{\Sigma}\left(H\Delta f+H(\overline{\rm Ric}(N,N)+|A|^{2})f-2H(H^{2}+\overline{K})f\right)dA\\
			&+2(\kappa-4\tau^{2})\int_{\Sigma}\left(2HC^{2}-\langle A(T),T\rangle\right)fdA\\
			=&\int_{\Sigma}\left(\Delta H+(\overline{\rm Ric}(N,N)+|A|^{2})H-2H(H^{2}+\overline{K})\right)fdA\\
			&+2(\kappa-4\tau^{2})\int_{\Sigma}\left(2HC^{2}-\langle A(T),T\rangle\right)fdA.
		\end{split}
	\end{equation}
	Consequently $x$ is a stationary point of the Willmore functional $\mathcal{W}$ if and only if
	\begin{equation}\label{eq:5.14}
		\Delta H+\left(|\Phi|^{2}+\overline{\rm Ric}(N,N)-2\overline{K}+4(\kappa-4\tau^{2})C^{2}\right)H-2(\kappa-4\tau^{2})\langle A(T),T\rangle=0.
	\end{equation}
	Finally, by an straightforward computation from~\eqref{eq:curv_E} and~\eqref{eq:2.4} we easily obtain
	\begin{equation}\label{eq:RicN}
		\overline{\rm Ric}(N,N)=\kappa-2\tau^{2}-(\kappa-4\tau^{2})C^{2},
	\end{equation}
	which jointly with~\eqref{eq:curvsect} and~\eqref{eq:5.14} yields the desired result.
\end{proof}

\begin{remark}
	It is not difficult to check that minimal tori and Hopf cylinders over a curve of geodesic curvature $k_{g}=\sqrt{2(2\tau^{2}-\kappa)}$, for all $\kappa,\tau\in\mathbb{R}$ with $\kappa<2\tau^{2}$ satisfy~\eqref{eq:5.14}. So, they are stationary points of the Willmore functional $\mathcal{W}$.
\end{remark}

Before presenting our classification result for Willmore surfaces in $\mathbb{E}^3(\kappa,\tau)$, we firstly need the following lemma whose proof follows the ideas developed in~\cite[Lemma 2.1]{Guo:04} (see also~\cite{Huisken:84}).

\begin{lemma}\label{lem:Kato-type}
	If $\Sigma^{2}$ is an isometrically immersed orientable surface into the homogeneous space $\mathbb{E}^{3}(\kappa,\tau)$, then
	\begin{equation}\label{eq:5.15}
		|\nabla A|^{2}\geq3|\nabla H|^{2}+2(\kappa-4\tau^{2})C\langle\nabla H,T\rangle.
	\end{equation}
\end{lemma}

\begin{proof}
	Given any $a\in\mathbb{R},$ let us consider the following tensor
	\begin{equation}
		\begin{split}
			F(X,Y,Z)=&\langle\nabla A(X,Y),Z\rangle\\
			&+a\left(\langle\nabla H,X\rangle\langle Y,Z\rangle+\langle\nabla H,Y\rangle\langle X,Z\rangle+\langle\nabla H,Z\rangle\langle X,Y\rangle\right).
		\end{split}
	\end{equation}
	A direct computation gives
	\begin{equation}
		F(X,Y,Z)^{2}=\langle\nabla A(X,Y),Z\rangle^{2}+2aQ_{1}(X,Y,Z)+a^{2}Q_{2}(X,Y,Z),
	\end{equation}
	where
	\begin{equation}
		\begin{split}
			Q_{1}(X,Y,Z)&=\langle\nabla A(X,Y),Z\rangle\left(\langle\nabla H,X\rangle\langle Y,Z\rangle+\langle\nabla H,Y\rangle\langle X,Z\rangle+\langle\nabla H,Z\rangle\langle X,Y\rangle\right)
		\end{split}
	\end{equation}
	and
	\begin{equation}
		\begin{split}
			Q_{2}(X,Y,Z)=&\left(\langle\nabla H,X\rangle^{2}\langle Y,Z\rangle^{2}+\langle\nabla H,Y\rangle^{2}\langle X,Z\rangle^{2}+\langle\nabla H,Z\rangle^{2}\langle X,Y\rangle^{2}\right)\\
			&+2\left(\langle\nabla H,X\rangle\langle Y,Z\rangle\langle\nabla H,Y\rangle\langle X,Z\rangle+\langle\nabla H,X\rangle\langle Y,Z\rangle\langle\nabla H,Z\rangle\langle X,Y\rangle\right)\\
			&+2\langle\nabla H,Y\rangle\langle X,Z\rangle\langle\nabla H,Z\rangle\langle X,Y\rangle.
		\end{split}
	\end{equation}
	In order to compute these last terms, let us take $\{e_{1},e_2\}$ an orthonormal frame on $\mathfrak{X}(\Sigma)$. Then, it is not difficult to check that
	\begin{equation}
		\sum_{i,j,k}^{2}\langle\nabla A(e_{i},e_{j}),e_{k}\rangle^{2}=|\nabla A|^{2}\quad\mbox{and}\quad\sum_{i,j,k}^{2}Q_{2}(e_{i},e_{j},e_{k})=12|\nabla H|^{2}.
	\end{equation}
	Besides that, from Codazzi equation and~\eqref{eq:traza_nA}, we have
	\begin{equation}
		\begin{split}
			\sum_{i,j,k}^{2}Q_{1}(e_{i},e_{j},e_{k})=&\sum_{i,j=1}^{2}\left(\langle\nabla A(e_{i},e_{j}),e_{j}\rangle\langle\nabla H,e_{i}\rangle+\langle\nabla A(e_{i},e_{j}),e_{i}\rangle\langle\nabla H,e_{j}\rangle\right)\\
			&+\sum_{i=1}^{2}\langle\nabla A(e_{i},e_{i}),\nabla H\rangle\\
			=&6|\nabla H|^{2}+2(\kappa-4\tau^{2})C\langle\nabla H,T\rangle.
		\end{split}
	\end{equation}
	Hence,
	\begin{equation}
		|F|^{2}=|\nabla A|^{2}+2a\left(6|\nabla H|^{2}+2(\kappa-4\tau^{2})C\langle\nabla H,T\rangle\right)+12a^{2}|\nabla H|^{2}.
	\end{equation}
	Taking $a=-1/2$ we obtain~\eqref{eq:5.15}.
\end{proof}

We can finally present our first main result.

\begin{theorem}\label{teo:Willmore surfaces}
	Let $\Sigma^{2}$ be an isometrically immersed orientable closed Willmore surface into an homogeneous space $\mathbb{E}^{3}(\kappa,\tau)$. Then,
	\begin{equation}
		\begin{split}
			\int_{\Sigma}&\left(|\Phi|^{4}-\left(2\tau^{2}-(\kappa-4\tau^{2})(1-3C^{2})\right)|\Phi|^{2}\right)dA\\
			&\qquad-(\kappa-4\tau^{2})\int_{\Sigma}\left(|\nabla C|^{2}+(K_{e}+\tau^{2})(1-5C^{2})+2\tau^{2}(1-3C^{2})\right)dA\geq0.
		\end{split}
	\end{equation}
	where the equality holds if and only if $\Sigma^{2}$ is a parallel surface.
	
	In particular, if $\kappa<2\tau^{2}$ the equality holds if and only if\, $\mathbb{E}^3(\kappa,\tau)=\mathbb{S}_b^3(\kappa,\tau)$ and $\Sigma^{2}$ is either a Clifford torus or a Hopf torus over a closed curve of geodesic curvature $\sqrt{2(2\tau^2-\kappa)}$ on $\mathbb{S}^{2}(\kappa)$.
\end{theorem}

\begin{proof}
	Firstly, {taking into account~\eqref{eq:2.38}}, \eqref{eq:2.35} can be written as follows,
	\begin{equation}
		\begin{split}
			\square(2H)&=4H\Delta H-2{\rm tr}(A\circ {\rm Hess}\,H)\\
			&=2H\Delta H-\dfrac{1}{2}\Delta|\Phi|^{2}-2|\nabla H|^{2}+\dfrac{1}{2}\Delta|A|^{2}-2{\rm tr}(A\circ {\rm Hess}\,H),
		\end{split}
	\end{equation}
	where $\Delta H^{2}=2H\Delta H+2|\nabla H|^{2}$ has been used. Consequently, by~\eqref{eq:2.32},
	\begin{equation}\label{eq:nuevasq2H}
		\begin{split}
			\square(2H)=&2H\Delta H-\dfrac{1}{2}\Delta|\Phi|^{2}+|\nabla A|^{2}-2|\nabla H|^{2}\\
			&+|\Phi|^{2}\left(2K_{e}+(\kappa-4\tau^{2})(5C ^{2}-1)+2\tau^{2}\right)\\
			&-2(\kappa-4\tau^{2})\left(H\langle\Phi( T), T\rangle+\tau\langle\Phi( T),J( T)\rangle\right).
		\end{split}
	\end{equation}
	Let us observe now that from Lemma~\ref{lem:Kato-type} we get
	\begin{equation}\label{eq:Kato-type-aux}
		\begin{split}
			|\nabla A|^2-2|\nabla H|^2&\geq |\nabla H|^2+2(\kappa-4\tau^2)C\langle \nabla H, T \rangle\\
			&\geq 2(\kappa-4\tau^2)C\langle \nabla H, T \rangle,
		\end{split}
	\end{equation}
	where the equality holds if and only if
	\begin{equation}\label{eq:Kato-type-parallel}
		|\nabla A|^2=3|\nabla H|^2+2(\kappa-4\tau^2)C\langle \nabla H,T\rangle=0,
	\end{equation}
	that is, if and only if $\Sigma^2$ is a parallel surface.
	{Then, from Lemma~\ref{divergencias:CY} and taking into account~\eqref{eq:nuevasq2H} and~\eqref{eq:Kato-type-aux} we obtain the following inequality,}
	\begin{equation}
		\begin{split}
			{\rm div}(P(2\nabla H))&=\square(2H)-2(\kappa-4\tau^{2})C\langle\nabla H,T\rangle\\
			&\geq2H\Delta H-\dfrac{1}{2}\Delta|\Phi|^{2}+|\Phi|^{2}\left(2K_{e}+(\kappa-4\tau^{2})(5C ^{2}-1)+2\tau^{2}\right)\\
			&\quad-2(\kappa-4\tau^{2})\left(H\langle\Phi(T),T\rangle+\tau\langle\Phi(T),J(T)\rangle\right).
		\end{split}
	\end{equation}
	Therefore, the divergence theorem yields
	\begin{equation}
		\begin{split}
			-2\int_{\Sigma}H\Delta HdA&\geq\int_{\Sigma}|\Phi|^{2}\left(2K_{e}+(\kappa-4\tau^{2})(5C^{2}-1)+2\tau^{2}\right)dA\\
			&\quad-2(\kappa-4\tau^{2})\int_{\Sigma}\left(H\langle\Phi(T),T\rangle+\tau\langle\Phi(T),J(T)\rangle\right)dA.
		\end{split}
	\end{equation}
	
	On the one hand, from Proposition~\ref{lem:1} we can write
	\begin{equation}\label{eq:HDeltaH}
		\begin{split}
			2H\Delta H&=-2H^{2}\left(|\Phi|^{2}+(\kappa-4\tau^{2})(1+C^{2})\right)+4(\kappa-4\tau^{2})H\langle A(T),T\rangle\\
			&=-\left(|\Phi|^{2}+2K_{e}\right)\left(|\Phi|^{2}+(\kappa-4\tau^{2})(3C^{2}-1)\right)+4(\kappa-4\tau^{2})H\langle\Phi(T),T\rangle\\
			&=-|\Phi|^{2}\left(2K_{e}+|\Phi|^{2}+(\kappa-4\tau^{2})(3C^{2}-1)\right)-2(\kappa-4\tau^{2})(3C^{2}-1)K_{e}\\&\qquad+4(\kappa-4\tau^{2})H\langle\Phi(T),T\rangle,
		\end{split}
	\end{equation}
	where we have used that
	\begin{equation}\label{eq:aux_normphi}
		|\Phi|^2-2H^2=-2K_e,
	\end{equation}
	which follows from~\eqref{eq:2.9} and~\eqref{eq:2.38}. Hence,
	\begin{equation}\label{eq:5.17}
		\begin{split}
			0\geq&\int_{\Sigma}|\Phi|^{2}\left(-|\Phi|^{2}+2(\kappa-4\tau^{2})C^{2}+2\tau^{2}\right)dA+2(\kappa-4\tau^{2})\int_{\Sigma}(1-3C^{2})K_{e}dA\\
			&\quad+2(\kappa-4\tau^{2})\int_{\Sigma}\left(H\langle\Phi(T),T\rangle-\tau\langle\Phi(T),J(T)\rangle\right)dA.
		\end{split}
	\end{equation}
	
	On the other hand, by using $A^2-2HA+K_eI=A^2-2H\Phi-\left(|\Phi|^2+K_e\right)I=0$ and the integrability equation~\eqref{eq:2.5}, we easily obtain
	\begin{equation}\label{eq:normnC2}
		|\nabla C|^2=2H\langle \Phi (T),T\rangle+2\tau\langle \Phi (T),J (T)\rangle+\left(|\Phi|^2+K_e+\tau^2\right)|T|^2.
	\end{equation}
	Now, let us consider the local orthonormal frame on $\mathfrak{X}(\Sigma)$, $\{e_1,e_2\}$ such that $e_1=\frac{T}{|T|}$ and $e_2=J(e_1)$ we get
	\begin{equation}\label{eq:div_JT_Aux}
		{\rm div}(J(T))=-\langle J(T),\nabla_{e_1}e_1\rangle+e_2(|T|),
	\end{equation}
	which using once more the integrability equations~\eqref{eq:2.5} yields
	\begin{equation}\label{eq:div_JT}
		{\rm div}(J(T))=2\tau C.
	\end{equation}
	So, from~\eqref{eq:2.5} and~\eqref{eq:div_JT},
	\begin{equation}\label{eq:div_JT1}
		{\rm div}(\tau CJ(T))=2\tau^{2}C^{2}-\tau\langle(A+\tau J)T,J(T)\rangle=-\tau\langle\phi(T),J(T)\rangle-\tau^{2}(1-3C^{2}).
	\end{equation}
	Thus, by~\eqref{eq:normnC2}
	\begin{equation}\label{eq:normnC21}
		\begin{split}
			2H\langle\Phi(T),T\rangle-2\tau\langle\phi(T),J(T)\rangle
			&=|\nabla C|^{2}+4{\rm div}(\tau CJ(T))+\tau^{2}(3-11C^{2})\\
			&\qquad-(|\Phi|^{2}+K_{e})(1-C^{2})
		\end{split}
	\end{equation}
	Therefore, by~\eqref{eq:5.17} we obtain
	\begin{equation}\label{eq:5.18}
		\begin{split}
			0\geq\int_{\Sigma}|\Phi|^{2}&\left(-|\Phi|^{2}+(\kappa-4\tau^{2})(3C^{2}-1)+2\tau^{2}\right)dA\\
			&\quad+(\kappa-4\tau^{2})\int_{\Sigma}\left(|\nabla C|^{2}+(K_{e}+\tau^{2})(1-5C^{2})+2\tau^{2}(1-3C^{2})\right)dA.
		\end{split}
	\end{equation}
	which is the desired inequality.
	
	Moreover, as we have remarked before, the equality holds in~\eqref{eq:5.18} if and only if $\Sigma^2$ is a closed parallel surface. Then, from Lemma~\ref{lem:4.1} $\Sigma^2$ is either a Hopf torus (necessarily in $\mathbb{S}^3_b(\kappa,\tau)$) over a Riemannian circle in $\mathbb{S}^2(\kappa)$, or a piece of a slice in $\mathbb{M}^2(\kappa)\times\mathbb{R}$. However, since $\Sigma^2$ is closed this last case only occurs in the case $\tau=0$ and $\kappa>0$, which do not satisfy the assumption $\kappa<2\tau^2$.
	
	Consequently, $\Sigma^2$ is a Hopf torus in $\mathbb{S}^3_b(\kappa,\tau)$, so in particular $C=0$ and $K_{e}=-\tau^{2}$. Hence,~\eqref{eq:HDeltaH} reads
	\begin{equation}
		0=\left(-|\Phi|^{2}+2\tau^{2}\right)(|\Phi|^{2}+\kappa-4\tau^{2}.
	\end{equation}
	Then, either $|\Phi|^{2}=2\tau^{2}$, which implies that $H=0$ and $\Sigma^{2}$ is the Clifford torus, or $|\Phi|^{2}+\kappa-4\tau^{2}=0$. Thus, from~\eqref{eq:aux_normphi} we get $H=\sqrt{\frac{2\tau^2-\kappa}{2}}$ and, consequently, $\Sigma^{2}$ is isometric to a Hopf torus in $\mathbb{S}^3_b(\kappa,\tau)$ over a curve of geodesic curvature $\sqrt{2\left(2\tau^{2}-\kappa\right)}$ on $\mathbb{S}^{2}(\kappa)$, for $0<\kappa<2\tau^{2}$.
\end{proof}

\section{Classification results for constant extrinsic curvature closed surfaces}\label{sec:5}

Let us begin by obtaining some new interesting divergence formulae, which will play a fundamental role in the proof of the main results in this section.

\begin{lemma}\label{divergencias}
	Let $\Sigma^{2}$ be an isometrically immersed surface into the homogeneous space $\mathbb{E}^{3}(\kappa,\tau)$. Then the following divergence formulae hold on $\Sigma^{2}$,
	\begin{itemize}
		\item[(a)] ${\rm div}(\nabla_{T}T)=K|T|^{2}+ T({\rm div}\, T)+|\nabla T|^{2}-4\tau^{2}C^{2}$.
		\item[(b)] ${\rm div}\left({\rm div}(T) T\right)= T({\rm div}\, T)+4H^{2}C^{2}.$
		\item[(c)] ${\rm div}\left(|T|\nabla|T|\right)=K|T|^{2}+ T({\rm div}\, T)+|\nabla T|^{2}-2\tau^{2}|T|^{2}-2\tau\langle \Phi(T),J(T)\rangle$.
	\end{itemize}
\end{lemma}

\begin{proof}
	Firstly, let us observe that items $(a)$ and $(b)$ have already been proved in~\cite{Torralbo:10.1} (see also~\cite[Lemma 3.2]{Hu:15}). However, we will include the proofs for the sake of completeness.
	
	From the integrability equations~\eqref{eq:2.5} it is immediate to check that
	\begin{equation}\label{eq:div_nTT}
		{\rm div} (\nabla_TT)={\rm div}(C(A-\tau J)(T))=-\langle A^2(T),T\rangle +\tau^2|T|^2+C{\rm div}(A(T))-\tau C{\rm div}(J(T)).
	\end{equation}
	
	On the one hand, given a local orthonormal frame $\{e_1,e_2\}$ on $\mathfrak{X}(\Sigma)$ diagonalizing $A$, from the Codazzi equation~\eqref{eq:2.11} it holds
	\begin{equation}\label{eq:div_AT}
		\begin{split}
			{\rm div}(A(T))&=\sum_{i=1}^2\langle (\nabla_{e_i}A)(T),e_i\rangle+ \sum_{i=1}^2\langle A(\nabla_{e_i}T),e_i\rangle\\
			&=\,{\rm tr}(\nabla_TA)+C(\kappa-4\tau^2)|T|^2+\sum_{i=1}^2\langle \nabla_{e_i}T,A(e_i)\rangle\\
			&=\,2T(H)+C(\kappa-4\tau^2)|T|^2+C|A|^2,
		\end{split}
	\end{equation}
	where in the last equality we have used again~\eqref{eq:2.5} and the fact that the trace commutes with the Levi-Civita connection. 
	
	On the other hand,~\eqref{eq:divT} yields
	\begin{equation}\label{eq:TdivT}
		T({\rm div}(T))=-2H\langle A(T),T\rangle+2CT(H).
	\end{equation} 
	Then, taking into account~\eqref{eq:div_JT},~\eqref{eq:div_AT} and~\eqref{eq:TdivT},~\eqref{eq:div_nTT} reads
	\begin{equation}\label{eq:div_nTT_2}
		{\rm div}(\nabla_TT)=K|T|^2+T({\rm div}(T))+C^2|A|^2-2\tau^2C^2,
	\end{equation}
	where we have used~\eqref{eq:2.12} and~\eqref{eq:2.29}. Finally, item $(a)$ follows by observing that
	\begin{equation}\label{eq:norm_nT}
		|\nabla T|^2=\sum_{i,j=1}^2\langle\nabla_{e_i}T,e_j\rangle^2= C^2(|A|^2+2\tau^2),
	\end{equation}
	for any $\{e_1,e_2\}$ local orthonormal frame on $\mathfrak{X}(\Sigma)$.
	
	Item $(b)$ follows directly from~\eqref{eq:divT}.
	
	With respect to item $(c)$, a direct computation from~\eqref{eq:2.5} guarantees us that
	\begin{equation}\label{eq:3.1}
		|T|\nabla|T|=C(A+\tau J)(T).
	\end{equation}
	Then, taking divergences in~\eqref{eq:3.1},
	\begin{equation}\label{eq:3.2}
		{\rm div}\left(|T|\,\nabla| T|\right)={\rm div}(A(T))C+\tau{\rm div}(J(T))C+\langle\nabla C,(A+\tau J) (T)\rangle.
	\end{equation}
	It is easy to check from~\eqref{eq:2.29} and from the integrability equations~\eqref{eq:2.5} that
	\begin{equation}\label{eq:3.4}
		\begin{split}
			\langle\nabla C,(A+\tau J) (T)\rangle&=-\langle A^{2}(T),T\rangle-2\tau\langle A(T),J(T)\rangle-\tau^{2}|T|^{2}\\
			&=-2H\langle A(T),T\rangle+K_{e}|T|^{2}-2\tau\langle \Phi(T),J(T)\rangle-\tau^{2}|T|^{2}.
		\end{split}
	\end{equation}
	
	Then, item $(c)$ follows by inserting~\eqref{eq:div_AT}-\eqref{eq:norm_nT} and~\eqref{eq:3.4} in~\eqref{eq:3.2}.	
\end{proof}

Bringing all these formulae together we get the desired divergenge-type formulae,

\begin{corollary}\label{cor:3.1}
	Let $\Sigma^{2}$ be an isometrically immersed surface into an homogeneous space $\mathbb{E}^{3}(\kappa,\tau)$. Then the following divergence formulae
	hold,
	\begin{equation}\label{eq:divW}
		\begin{split}
			{\rm div}\,(\mathcal{U})&=\Delta K_{e}+|\nabla A|^{2}-4|\nabla H|^{2}+2|\Phi|^{2}\left(K_{e}+(\kappa-4\tau^{2})(4C^{2}-1)+\tau^{2}\right)\\
			&\quad-2(\kappa-4\tau^{2})\left(2H\langle\Phi( T), T\rangle+(K_{e}-\tau^{2})(1-3C^{2})\right),
		\end{split}
	\end{equation}
	where $\mathcal{U}=P(2\nabla H)+(\kappa-4\tau^{2})\left( \nabla_{ T} T-| T|\nabla| T|+{\rm div}\,( T) T \right)$
	and
	\begin{equation}
		\begin{split}\label{eq:divXTF2}
			{\rm div}\,(\mathcal{V})&=\Delta K_{e}+|\nabla A|^{2}-4|\nabla H|^{2}+2|\Phi|^{2}\left(K_{e}+3(\kappa-4\tau^{2})C^2+\tau^{2}\right)\\
			&\quad-2(\kappa-4\tau^{2})\left( |\nabla C|^2-2(K_e+\tau^2)C^2 \right),
		\end{split}
	\end{equation}
	where $\mathcal{V}=P(2\nabla H)+(\kappa-4\tau^2)\left(|T|\nabla |T|+{\rm div}(T)T-\nabla_TT\right)$.
\end{corollary}

\begin{proof}
	On the one hand, let $\mathcal{U}_1=\nabla_{ T} T-| T|\nabla| T|+{\rm div}\,( T) T $, then from items $(a)$, $(b)$ and $(c)$ of Lemma~\ref{divergencias} we can compute
	\begin{equation}\label{eq:3.12}
		{\rm div}\,(\mathcal{U}_1)=-2\tau^{2}(3C^{2}-1)+2\tau\langle\Phi( T),J( T)\rangle+ T({\rm div}( T))+4H^{2}C^{2}.
	\end{equation}
	Then, from~\eqref{eq:3.12},~\eqref{eq:TdivT} and item $(d)$ in Lemma~\ref{divergencias} we get
	\begin{equation}\label{eq:divW_aux1}
		\begin{split}
			{\rm div}(\mathcal{U})=&\square(2H)-2H(\kappa-4\tau^{2})\left(\langle A( T), T\rangle-2C^{2}H\right)\\
			&-2\tau(\kappa-4\tau^{2})\left(\tau(3C^{2}-1)-\langle\Phi( T),J( T)\rangle\right).
		\end{split}
	\end{equation}
	Taking now into account Proposition~\ref{prop:2.1} jointly with~\eqref{eq:2.4} and the definition of $\Phi$ we easily deduce
	\begin{equation}\label{eq:divW_aux2}
		\begin{split}
			{\rm div}(\mathcal{U})&=\,\Delta K_{e}+|\nabla A|^{2}-4|\nabla H|^{2}+2|\Phi|^{2}\left(K_{e}+(\kappa-4\tau^{2})(4C^{2}-1)+\tau^{2}\right)\\
			&\quad+(\kappa-4\tau^{2})\left(-4H\langle\Phi( T), T\rangle+(1-3C^{2})(2\tau^2+|\Phi|^2-2H^2)\right).
		\end{split}
	\end{equation}
	Then~\eqref{eq:divW} follows from \eqref{eq:aux_normphi}.
	
	On the other hand, let us observe that 
	\begin{equation}\label{eq:difWV}
		\mathcal{V}-\mathcal{U}=2(\kappa-4\tau^2)(|T|\nabla |T|-\nabla_TT).
	\end{equation}
	Therefore, from~\eqref{eq:divW} and items $(a)$ and $(c)$ of Lemma~\ref{divergencias}, it holds
	\begin{equation}\label{eq:divWV}
		\begin{split}
			\textrm{div}(\mathcal{V})&=\Delta K_e +|\nabla A|^2-4|\nabla H|^2+2|\Phi|^2\left( K_e+(\kappa-4\tau^2)(4C^2-1)+\tau^2 \right)\\
			&\quad+2(\kappa-4\tau^2)\left(4\tau^2 C^2-2\tau^2|T|^2-2\tau\langle\Phi(T),J(T)\rangle-2H\langle \Phi(T),T \rangle-(K_e-\tau^2)(1-3C^2)\right).
		\end{split}
	\end{equation}
	Then, the desired formula~\eqref{eq:divXTF2} follows from~\eqref{eq:2.4} and~\eqref{eq:normnC2}, so Corollary~\ref{cor:3.1} is proved.
\end{proof}

In the next results we will approach the case in which the extrinsic curvature is constant and negative. For this, the following lemma is essential.

\begin{lemma}\label{lem:4.2}
	Let $\Sigma^{2}$ be an isometrically immersed orientable surface into the homogeneous space $\mathbb{E}^{3}(\kappa,\tau)$ with constant extrinsic curvature $K_{e}<0$. Then 
	\begin{equation}\label{reverse inequality}
		|\nabla A|^{2}\leq4|\nabla H|^{2}.
	\end{equation}
	In particular, the equality holds if and only if $\Sigma^{2}$ is a parallel surface.
\end{lemma}

\begin{proof}
	Indeed, let $\{e_{1},e_{2}\}$ be a local orthonormal frame which diagonalizes $A$, that is, $A(e_i)=\lambda_ie_i$, $i=1,2$. Then $|\nabla A|^{2}=\sum_{i,j=1}^2 (e_i(\lambda_j))^2$, and by a direct computation we get
	\begin{equation}\label{eq:4.2}
		4|\nabla H|^{2}=(e_1(\lambda_{1})+e_1(\lambda_2))^{2}+(e_2(\lambda_{1})+e_2(\lambda_2))^{2}.
	\end{equation}
	Hence,
	\begin{equation}\label{eq:4.3}
		|\nabla A|^{2}-4|\nabla H|^{2}=-2(e_1(\lambda_1)e_1(\lambda_2)+e_2(\lambda_1)e_2(\lambda_2)).
	\end{equation}
	On the other hand, since $K_{e}=\lambda_{1}\lambda_{2}$ is a negative constant, taking derivatives with respect to $e_{1}$ and $e_{2}$,
	\begin{equation}\label{eq:deriv_Ke}
		0=e_i(K_e)=e_i(\lambda_1)\lambda_2+\lambda_1e_i(\lambda_2), \quad i=1,2.
	\end{equation}
	Furthermore, $\lambda_1,\lambda_2\neq 0$, so from~\eqref{eq:deriv_Ke} it holds
	\begin{equation}\label{eq:deriv_Ke_2}
		e_i(\lambda_1)=-\frac{\lambda_1}{\lambda_2}e_i(\lambda_2),\quad i=1,2.
	\end{equation}
	Therefore,~\eqref{eq:4.3} reads
	\begin{equation}\label{eq:4.6}
		|\nabla A|^{2}-4|\nabla H|^{2}=\dfrac{2\lambda_{1}}{\lambda_{2}}(e_1^2(\lambda_2)+e_2^2(\lambda_2))=\dfrac{2K_{e}}{\lambda_{2}^{2}}(e_1^2(\lambda_2)+e_2^2(\lambda_2))\leq 0
	\end{equation}
	as desired. The conclusion about the equality is immediate.
\end{proof}

\begin{corollary}\label{cor:4.1}
	There exists no immersed surface into the homogeneous space $\mathbb{E}^{3}(\kappa,\tau)$ with $\kappa-4\tau^2\neq0$, satisfying the equality in~\eqref{reverse inequality} and having positive constant extrinsic curvature.
\end{corollary}

\begin{proof}
	Indeed, suppose there exists an immersed surface $\Sigma^{2}$ into $\mathbb{E}^{3}(\kappa,\tau)$, $\kappa-4\tau^2\neq 0$, satisfying the equality in~\eqref{reverse inequality} and having positive constant extrinsic curvature. Following the same reasoning as in the proof of Lemma~\ref{lem:4.2} we obtain~\eqref{eq:4.6}, so
	\begin{equation}\label{eq:4.6.1}
		0=|\nabla A|^{2}-4|\nabla H|^{2}=\dfrac{2K_{e}}{\lambda_{2}^{2}}(e_1^2(\lambda_2)+e_2^2(\lambda_2))\geq 0.
	\end{equation}
	Since $K_{e}>0$, we must have $e_1(\lambda_2)=e_2(\lambda_2)=0$. Therefore $\lambda_2$ is constant, so by the assumption on the extrinsic curvature $\lambda_1$ is also constant. Thus, $\Sigma^{2}$  should be a parallel surface of $\mathbb{E}^{3}(\kappa,\tau)$. Hence, from Lemma~\ref{lem:4.1} $\Sigma^{2}$ is either isometric to a piece of a Hopf cylinder or of a slice, which is a contradiction since in both cases $K_{e}=-\tau^{2}\leq 0$.
\end{proof}

Now, we present our first result related to surfaces with constant extrinsic curvature in $\mathbb{S}_b^3(\kappa,\tau)$.

\begin{theorem}\label{teo:1.1}
	Let $\Sigma^{2}$ be an isometrically immersed closed surface into the homogeneous space $\mathbb{E}^{3}(\kappa,\tau)$, $\kappa-4\tau^2\neq 0$, with negative constant extrinsic curvature. Then
	\begin{equation}\label{Simons}
		\int_{\Sigma}\!|\Phi|^{2}\!\left(K_{e}+(\kappa-4\tau^{2})(4C ^{2}-1)+\tau^{2}\right)dA\,\geq(\kappa-4\tau^{2})\!\int_{\Sigma}Q_{\tau,K_{e}}dA,
	\end{equation}
	where
	\begin{equation}\label{Q}
		Q_{\tau,K_{e}}=2H\langle\Phi( T), T\rangle+(K_{e}-\tau^{2})(1-3C^{2}).
	\end{equation}
	The equality holds if and only if $\mathbb{E}^3(\kappa,\tau)=\mathbb{S}_b^3(\kappa,\tau)$ and $\Sigma^{2}$ is a Hopf torus over a Riemannian circle in $\mathbb{S}^2(\kappa)$.
\end{theorem}

\begin{proof}
	By Corollary~\ref{cor:3.1},
	\begin{equation}\label{eq:divW_thm1}
		\begin{split}
			{\rm div}\,(\mathcal{U})&=\,\Delta K_{e}+|\nabla A|^{2}-4|\nabla H|^{2}+2|\Phi|^{2}\left(K_{e}+(\kappa-4\tau^{2})(4C^{2}-1)+\tau^{2}\right)\\
			&\quad-2(\kappa-4\tau^{2})\left(2H\langle\Phi( T), T\rangle+(K_{e}-\tau^{2})(1-3C^{2})\right).
		\end{split}
	\end{equation}
	Since we are supposing that the extrinsic curvature is a negative constant, from Lemma~\ref{lem:4.2}, we can estimate the divergence in this way
	\begin{equation}\label{eq:divW_thm2}
		\begin{split}
			{\rm div}\,(\mathcal{U})&\leq2|\Phi|^{2}\left(K_{e}+(\kappa-4\tau^{2})(4C^{2}-1)+\tau^{2}\right)\\
			&\quad-2(\kappa-4\tau^{2})\left(2H\langle\Phi( T), T\rangle+(K_{e}-\tau^{2})(1-3C^{2})\right).
		\end{split}
	\end{equation}
	Therefore, taking integrals and using the classical divergence theorem, we have
	\begin{equation}\label{eq:int_thm}
		\int_{\Sigma}\left\{|\Phi|^{2}\left(K_{e}+(\kappa-4\tau^{2})(4C^{2}-1)+\tau^{2}\right)-(\kappa-4\tau^{2})Q_{\tau,K_{e}}\right\}dA\geq0,
	\end{equation}
	where $Q_{\tau,K_{e}}$ is defined as in~\eqref{Q}, which is the desired inequality.
	
	Furthermore, the equality is satisfied if and only if the equality holds in~\eqref{reverse inequality}. Since $K_{e}<0$, Lemma~\ref{lem:4.2} guarantees that $\Sigma^{2}$ is a parallel surface in $\mathbb{E}^{3}(\kappa,\tau)$. Therefore, from Lemma~\ref{lem:4.1}, we conclude that $\Sigma^{2}$ is isometric to a piece of a Hopf cylinder or to a slice of $\mathbb{M}^2(\kappa)\times\mathbb{R}$ when $\tau=0$. Thus, by closedness and recalling that slices in $\mathbb{M}^2(\kappa)\times\mathbb{R}$ are totally geodesic surfaces, so consequently satisfy $K_e=0$, the equality in~\eqref{eq:int_thm} is only satified in the case where  $\mathbb{E}^3(\kappa,\tau)=\mathbb{S}_b^3(\kappa,\tau)$ and $\Sigma^2$ is isometric to a Hopf torus.
\end{proof}

We can also obtain the following alternative characterization result from~\eqref{eq:divXTF2}.

\begin{theorem}\label{teo:1.2}
	Let $\Sigma^{2}$ be an isometrically immersed closed surface with negative constant extrinsic curvature into the homogeneous space $\mathbb{E}^{3}(\kappa,\tau)$ such that $\kappa-4\tau^2>0$. Then
	\begin{equation}\label{Simons2}
		\int_{\Sigma}\left\{ \left(3(\kappa-4\tau^2)C^2+K_e+\tau^2\right)|\Phi|^2+2(\kappa-4\tau^2)(K_e+\tau^2)C^2\right\} dA\,\geq 0.
	\end{equation}
	The equality holds if and only if $\mathbb{E}^3(\kappa,\tau)=\mathbb{S}_b^3(\kappa,\tau)$ and $\Sigma^{2}$ is a Hopf torus over a Riemannian circle in $\mathbb{S}^2(\kappa)$.
\end{theorem}
\begin{proof}
	The proof of~\eqref{Simons2} follows immediately taking integrals in~\eqref{eq:divXTF2} and taking into account Lemma~\ref{lem:4.2}. The conclusion regarding the equality follows as in Theorem~\ref{teo:1.1}.
\end{proof}

\section*{Acknowledgements}
	The authors would like to heartily thank the referee for his/her valuable remarks and comments. The first author is partially supported by MICINN/FEDER project PGC2018-097046-B-I00, by the Regional Government of Andalusia ERDEF project PY20-01391 and Fundaci\'on S\'eneca project 19901/GERM/15, Spain. Her work is a result of the activity developed within the framework of the Program in Support of Excellence Groups of the Regi\'{o}n de Murcia, Spain, by Fundaci\'{o}n S\'{e}neca, Science and Technology Agency of the Regi\'{o}n de Murcia. The second author is also partially supported by CNPq, Brazil, under the grant 431976/2018-0.


\begin{thebibliography}{n}
	
	\bibitem{Abresch:04} U. Abresch and H. Rosenberg,
	{\em A Hopf differential for constant mean curvature surfaces in $\mathbb{S}^{2}\times\mathbb{R}$ and $\mathbb{H}^{2}\times\mathbb{R}$},
	Acta Math. {\bf 193} (2004), 141--174.
	
	\bibitem{Aledo:05} J.A. Aledo, L.J. Al\'{\i}as and A. Romero,
	{\em A new proof of Liebmann classical rigidity theorem for surfaces in space forms},
	Rocky Mountain J. Math. {\bf 35} (2005), 1811--1824.
	
	\bibitem{Alexandrov:56} A.D. Alexandrov, 
	{\em Uniqueness theorems for surfaces in the large I}, 
	Vestnik Leningrad Univ. {\bf 11} (1956), 5--17.
	
	\bibitem{Colares:97} J.L.M. Barbosa and A.G. Colares, 
	{\em Stability of hypersurfaces with constant $r$-mean curvature}, 
	Ann. Global Anal. Geom. {\bf 15} (1997), 277--297.
	
	\bibitem{do Carmo:88} J.L. Barbosa, M. do Carmo and J. Eschenburg, 
	{\em Stability of hypersurfaces of constant mean curvature in Riemannian manifolds}, 
	Math. Z. {\bf197}, (1988) 123--138.
	
	\bibitem{Dillen:02} M. Belkhelfa, F. Dillen and J. Inoguchi, 
	{\em Surfaces with parallel second fundamental form in Bianchi-Cartan-Vranceanu spaces}, 
	in: PDE’s, submanifolds and affine differential geometry, Banach Center Publ., Polish Acad. Sci. Inst. Math., Warsaw, {\bf57} (2002), 67--87.
	
	\bibitem{Bryant:84} R.L. Bryant,
	{\em A duality theorem for Willmore surfaces}, 
	J. Differential Geom. {\bf20} (1984), 23--53
	
	\bibitem{Cao:07} L. Cao and H.Li, 
	{\em $r$-minimal submanifolds in space forms}, 
	Ann. Global Anal. Geom. {\bf32} (2007), 311-341.
	
	\bibitem{Cheng-Yau:77} S.Y. Cheng and S.T. Yau,
	{\em Hypersurfaces with constant scalar curvature},
	Math. Ann. {\bf225} (1977), 195--204.
	
	\bibitem{Daniel:07} B. Daniel, 
	{\em Isometric immersions into $3$-dimensional homogeneous manifolds}, 
	Comment. Math. Helv. {\bf82} (2007), 87--131.
	
	\bibitem{dos Santos:18} F.R. dos Santos,
	{\em Rigidity of surfaces with constant extrinsic curvature in the Riemannian product spaces}, Bull. Braz. Math. Soc, New Series {\bf 52} (2021), 307--326.
	
	\bibitem{Galvez:08} J. G\'alvez, A. Mart\'{i}nez and P. Mira, 
	{\em The Bonnet problem for surfaces in homogeneous 3-manifolds}, 
	Comm. Anal. Geom. {\bf16} (2008), 907--935.
	
	\bibitem{Guo:04} Z. Guo,
	{\em Willmore submanifolds in the unit sphere}
	Collect. Math. {\bf55} (2004), 279-287
	
	\bibitem{Hilbert:01} D. Hilbert, 
	{\em \"{U}ber Fl\"{a}chen von konstanter Gau{\ss}scher Kr\"{u}mmung}, 
	Trans. Amer. Math. Soc. {\bf 2} (1901), 87--99. 
	
	\bibitem{Hopf:83} H. Hopf, 
	{\em Differential Geometry in the large}, 
	Lecture Notes in Math., 1000, Springer-Verlag, Berlin, 1983.
	
	\bibitem{Hu:15} Z. Hu, D. Lyu and J. Wang,
	{\em On rigidity phenomena of compact surfaces in homogeneous 3-manifolds},
	Proc. Amer. Math. Soc. {\bf143} (2015), 3097--3109.
	
	\bibitem{Huisken:84} G. Huisken, 
	{\em Flow by mean curvature of convex surfaces into spheres}, 
	J. Differential Geom. {\bf20} (1984), 237-266
	
	\bibitem{Liebmann:1899} H. Liebmann, 
	{\em Eine neue eigenschaft der kugel}, 
	Nach. Kgl. Ges. Wiss. G\"{o}ttingen, Math.-Phys. Klasse (1899), 44--55.
	
	\bibitem{Coda:12} F.C. Marques and A. Neves, 
	{\em Min-max theory and the Willmore conjecture}, 
	Ann. of Math. {\bf179} (2014), 683--782.
	
	\bibitem{Meeks:04} W.H. Meeks and H. Rosenberg,
	{\em Stable minimal surfaces in $M\times\mathbb{R}$}, 
	J. Differential Geom. {\bf68} (2004), 515--534.
	
	\bibitem{Montiel:91} S. Montiel and A. Ros,
	{\em Compact hypersurfaces: the Alexandrov theorem for higher order mean curvatures},
	in Differential geometry, Pitman Monogr. Surveys Pure Appl. Math., vol. 52, Longman Sci. Tech., Harlow, 1991, pp. 279--296.
	
	\bibitem{Nomizu:69} K. Nomizu and B. Smyth,
	{\em A formula of Simons' type and hypersurfaces with constant mean curvature},
	J. Differential Geom. {\bf3} (1969), 367--377.
	
	\bibitem{O'Neill:83} B. O'Neill,
	{\em Semi-Riemannian Geometry, with Applications to Relativity},
	New York: Academic Press (1983).
	
	\bibitem{Pampano:20} A. P\'ampano, {\em Critical tori for mean curvature energies in Killing submersions}, Nonlinear Anal. {\bf 200} (2020), 112092, 18 pp.
	
	\bibitem{Rosenberg:02} H. Rosenberg,
	{\em Minimal surfaces in $M^{2}\times\mathbb{R}$}, Illinois J. Math {\bf46} (2002), 1177--1195.
	
	\bibitem{Tribuzy:12} H. Rosenberg and R. Tribuzy,
	{\em Rigidity of convex surfaces in the homogeneous spaces},
	Bull. Sci. Math. {\bf136} (2012), 892--898.
	
	\bibitem{Toubiana:} R. Souam and E. Toubiana,
	{\em Totally umbilic surfaces in homogeneous $3$-manifolds},
	Comment. Math. Helv. {\bf84} (2009), 673--704.
	
	\bibitem{Thurston:97} W.M. Thurston,
	{\em Three-dimensional geometry and topology Vol. I},
	Princeton Mathematical Series, Vol. 35. Princeton University Press, 1997.
	
	\bibitem{Torralbo:10.1} F. Torralbo and F. Urbano,
	{\em On the Gauss curvature of closed surfaces in homogeneous 3-manifolds},
	Proc. Amer. Math. Soc. {\bf138} (2010), 2561--2567.
	
	\bibitem{Torralbo:12.1} F. Torralbo and F. Urbano,
	{\em Compact stable constant mean curvature surfaces in homogeneous 3-manifolds},
	Indiana Univ. Math. J. {\bf61} (2012), 1129--1156.
	
	\bibitem{Weiner:78} J.L. Weiner, 
	{\em On a problem of Chen, Willmore, et al}, 
	Indiana Univ. Math. J. {\bf27} (1978), 19--35.
	
\end{thebibliography}
\end{document}